\documentclass[10pt]{article}

\usepackage[hmargin=35mm,vmargin=35mm]{geometry}

\usepackage{amsmath}              
\usepackage{amssymb}              
\usepackage{amsthm}               
\usepackage{enumerate}            
\usepackage{graphicx}             
\usepackage{mathrsfs}             
\usepackage[tight]{subfigure}     
\usepackage{url}                  

\newtheorem{theorem}{Theorem}
\newtheorem{lemma}[theorem]{Lemma}

\newtheorem{algorithm}[theorem]{Algorithm}

\theoremstyle{definition}

\newcommand{\cc}[1]{C(#1)}

\newcommand{\arbcone}{\mathscr{P}}
\newcommand{\hilb}[1]{\mathrm{Hilb}(#1)}
\newcommand{\intcone}[2][7n]{#2 \cap \Z^{#1}}
\newcommand{\knotinfo}{\emph{KnotInfo}}
\newcommand{\normaliz}{\emph{Normaliz}}

\newcommand{\R}{\mathbb{R}}

\newcommand{\regina}{\emph{Regina}}
\newcommand{\scone}{\mathscr{C}}

\newcommand{\tri}{\mathcal{T}}
\newcommand{\vrep}[1]{\mathbf{v}(#1)}
\newcommand{\Z}{\mathbb{Z}}

\title{Enumerating fundamental normal surfaces: \\
    Algorithms, experiments and invariants}
\author{Benjamin A.~Burton}

\date{30 September 2013}

\begin{document}

\maketitle

\begin{abstract}
    Computational knot theory and 3-manifold topology have seen significant
    breakthroughs in recent years, despite the fact that many key algorithms
    have complexity bounds that are exponential or greater.  In this setting,
    experimentation is essential for understanding the limits of
    practicality, as well as for gauging the relative merits of competing
    algorithms.

    In this paper we focus on normal surface theory, a key tool that appears
    throughout low-dimensional topology.  Stepping beyond the well-studied
    problem of computing vertex normal surfaces (essentially extreme rays of
    a polyhedral cone), we turn our attention to the more complex task of
    computing fundamental normal surfaces (essentially an integral basis for
    such a cone).  We develop, implement and experimentally compare a primal
    and a dual algorithm, both of which combine domain-specific techniques
    with classical Hilbert basis algorithms.  Our experiments indicate that we
    can solve extremely large problems that were once though intractable.
    As a practical application of our techniques, we fill gaps from the
    KnotInfo database by computing 398 previously-unknown crosscap numbers
    of knots.

\end{abstract}

%
%

\section{Introduction}

Efficient computation in knot theory and 3-manifold topology remains a
significant challenge, with important implications for mathematical
research in these fields.

Theoretically, many algorithms in these fields
are highly complex and appear to be infeasibly slow.
For instance, the algorithm to test whether two 3-manifolds are
homeomorphic (topologically equivalent) has never been implemented, and
explicit bounds on its running time have never been computed;
see \cite{matveev03-algms} for just some of some of its highly intricate
components.

Practically, however, the right blend of topology, algorithms and heuristics can
yield surprising results.  Well known examples include the software packages
\emph{SnapPea} \cite{snappy,snappea}, \emph{Regina} \cite{regina},
and the \emph{3-Manifold Recogniser} \cite{recogniser},
which are extremely effective
for solving geometric, decomposition and recognition problems.
This is despite the fact that there are often no proofs to guarantee the
efficiency (or in some cases even the termination) of the underlying algorithms.

Obtaining practical algorithms such as these---even without theoretical
guarantees---is important for mathematicians.
A striking example was the resolution after 30 years of
Thurston's long-standing question on the Weber-Seifert
dodecahedral space \cite{burton12-ws}, in part due to
surprisingly effective heuristic improvements to an algorithm that
had been known since the 1980s \cite{jaco84-haken}, and whose
best theoretical complexity bound remains doubly-exponential
even today \cite{burton13-large}.

This is a typical pattern:
where theoretical bounds do exist in low-dimensional topology, they are
often extremely large---sometimes exponential, sometimes
tower-of-exponential---and they often bear little relation to practical
running times.
In this setting, experimentation plays two crucial roles:
\begin{itemize}
\item It is necessary for understanding the practical boundary between
what is and is not feasible;
\item It is essential when comparing competing algorithms, since
algorithms with smaller theoretical bounds do not always
performs better in practice.
\end{itemize}

In this paper we focus on algorithms that enumerate \emph{normal surfaces}.
Normal surfaces are ubiquitous in algorithmic 3-manifold topology and
knot theory, which gives such algorithms widespread
topological applications.  We show one such application at the end
of this paper, where we apply our techniques to compute
previously-unknown invariants of knots.

The main idea of normal surface theory is to
reduce difficult topological searches
to discrete problems on integer points in rational cones
\cite{hass99-knotnp,matveev03-algms}.  A typical topological algorithm
that uses normal surfaces might run as follows:
\begin{itemize}
    \item
    We present the input to our problem as a 3-manifold
    triangulation $\tri$ (so, if the input is a knot, we triangulate
    the \emph{knot complement} \cite{hass99-knotnp}), and we frame our
    topological problem as a search for some embedded surface in $\tri$ with a
    particular property (e.g., for the unknot recognition problem
    we search for an embedded disc with non-trivial boundary
    \cite{hass99-knotnp}).
    \item We prove that we can restrict this
    search to \emph{normal surfaces} in $\tri$, which are properly embedded
    surfaces that intersect the tetrahedra of $\tri$ in a well-behaved fashion.
    We encode normal surfaces arithmetically as
    integer points in the \emph{normal surface solution cone} $\scone$,
    which is a pointed rational cone in $\R^{7n}$ where
    $n$ is the number of tetrahedra in $\tri$.
    \item
    We prove that there is a \emph{finite} list
    $\mathscr{L}$ of points in this cone for which,
    if $\tri$ has a surface with the desired property,
    then some such surface is described by a point in our list
    $\mathscr{L}$.
    The algorithm now builds the cone $\scone$, constructs the
    finite list $\mathscr{L}$, rebuilds the corresponding surfaces,
    and tests each for the desired property.
\end{itemize}
For different topological algorithms, the list $\mathscr{L}$ takes on
different forms:
\begin{enumerate}[(i)]
    \item For some algorithms, $\mathscr{L}$ contains the smallest
    integer point on each extremal ray of the cone $\scone$.
    The corresponding surfaces are called \emph{vertex normal surfaces}.

    Topological algorithms of this type include
    unknot recognition \cite{haken61-knot,hass99-knotnp},
    3-sphere recognition \cite{jaco03-0-efficiency,rubinstein95-3sphere},
    connected sum decomposition \cite{jaco03-0-efficiency}, and
    Hakenness testing \cite{jaco84-haken,jaco95-algorithms-decomposition}.

    \item For some algorithms, $\mathscr{L}$ is the Hilbert basis for
    the cone $\scone$; that is, all integer points in $\scone$ that cannot be
    expressed as a non-trivial sum of other integer points in $\scone$.
    The corresponding surfaces are called \emph{fundamental normal surfaces}.

    Algorithms of this type include JSJ decomposition \cite{matveev03-algms},
    computing knot genus in arbitrary 3-manifolds
    \cite{agol02-knotgenus,schubert61-normal},
    computing the crosscap number of a knot \cite{burton12-crosscap}, and
    determining which Dehn fillings of a knot complement are Haken
    \cite{jaco03-decision}.
    %

    \item For some algorithms, the list $\mathscr{L}$ is much larger
    again: typically one must first enumerate all fundamental normal
    surfaces, and then (often at great expense) extend this list
    in some way.
\end{enumerate}

We emphasise that not all integer points in $\scone$ correspond to
normal surfaces: there are additional combinatorial constraints called
the \emph{quadrilateral constraints}
that such points must satisfy.  As a result, normal surfaces only
represent a very small subset of the integer points in $\scone$.
This is a powerful observation that underpins much of this paper.

Enumerating vertex surfaces---case~(i) above---is a ``best case scenario'',
and is now a well-studied problem with highly effective solutions
\cite{burton10-dd,burton13-tree}.  The running time remains exponential,
but this is unavoidable since the output size is known to be
exponential for some families of inputs
\cite{burton10-complexity,burton13-bounds}.

Enumerating fundamental surfaces---case~(ii)---is now the next major step
to be made in ``practical'' normal surface theory, and this is what we
address in this paper.
We emphasise that, because of the high dimensionality of the cone
and the potential for super-exponentially many solutions
\cite{hass99-knotnp}, ``out-of-the-box'' Hilbert basis algorithms
such as those found in the excellent software package
\normaliz\ \cite{bruns10-normaliz,bruns01-closure}
are impractically slow for all but the simplest inputs.
Instead we must marry classical Hilbert basis algorithms
with the structure and constraints of normal surface theory.
The result of this marriage, as described in this paper,
is the first practical software
for enumerating fundamental normal surfaces.

It is widely noted that no one algorithm for Hilbert basis
enumeration is superior in all settings
\cite{bruns10-normaliz,contejean94-incremental,pasechnik01-elliot}.
For this reason, we develop and experimentally compare two algorithms,
which at their core are based upon two modern Hilbert basis algorithms:
\begin{itemize}
    \item Our \emph{primal algorithm} draws upon the
    primal Hilbert basis algorithm described by
    Bruns, Ichim and Koch \cite{bruns10-normaliz,bruns01-closure}.
    In brief, we (i)~enumerate all vertex normal surfaces;
    (ii)~use these to build a piecewise-convex representation of the
    non-convex portion of $\scone$ that satisfies the quadrilateral
    constraints; and then
    (iii)~run the Bruns-Ichim-Koch algorithm over each convex
    piece and combine the results.

    \item Our \emph{dual algorithm} is based on the
    dual Hilbert basis algorithm of Bruns and Ichim
    \cite{bruns10-normaliz}, based on earlier work by Pottier
    \cite{pottier96-euclidean}.
    This is an inductive algorithm that generalises the double
    description method for enumerating extreme rays of a polyhedron
    \cite{burton10-dd,fukuda96-doubledesc,motzkin53-dd}.
    Our algorithm closely follows the general Bruns-Ichim method,
    with an additional ``filtering'' step that
    enforces the quadrilateral constraints at each intermediate stage.
\end{itemize}

Through detailed experimentation over a census of 10,986 triangulations
whose cones have up to 245 dimensions, we find that the primal algorithm
is both faster and more consistent in performance.  As well as comparing
the primal and dual algorithms, we use our experiments to study the
bottlenecks of the primal algorithm, and thereby identify avenues for
future improvement.

As an illustration of its applicability, we use the primal algorithm to
compute 398 previously-unknown crosscap numbers of knots.
The crosscap number is a knot invariant that is related to, but offers
distinct information from, the knot genus
\cite{clark78-crosscaps,murakami95-crosscap}.
There is still no general algorithm known for computing crosscap
numbers, and our results go a significant way towards filling in the
2640 missing crosscap numbers from the online \knotinfo\ database of
knot invariants.


All implementations use the freely-available software
package {\regina} \cite{burton04-regina,regina},
and the primal algorithm also incorporates portions of {\normaliz}~2.10.1.
The primal and dual algorithms are now built directly into {\regina},
and the additional code for computing crosscap numbers can be downloaded
from \url{http://www.maths.uq.edu.au/~bab/code/}.


%
%

\section{Preliminaries} \label{s-prelim}

Here we give a very brief overview of triangulations and normal surfaces.
In keeping with the focus of this paper, we concentrate on the algebraic
formulation at the expense of geometric insight; for further information
the reader is referred to the excellent summary in \cite{hass99-knotnp}.

Throughout this paper, the input for a topological problem is a
\emph{3-manifold triangulation} $\tri$, built from $n$ tetrahedra by
affinely identifying
(or ``gluing together'') some or all of the $4n$ tetrahedron faces in pairs
so that the resulting topological space is a 3-manifold
(possibly with boundary).

Such triangulations are more general than simplicial complexes:
as a result of the face gluings, different edges of the same tetrahedron
might be identified together, and likewise with vertices.  Two faces of
the same tetrahedron may even be glued together.  This general definition
allows us to express rich topological
structures using few tetrahedra, which is important for computation.

A \emph{normal surface} in $\tri$ is a properly
embedded surface that meets each tetrahedron in a (possibly
empty) disjoint union of \emph{normal discs}, each of which is a curvilinear
triangle or quadrilateral.
Each triangle separates one vertex of the tetrahedron from the others,
and each quadrilateral separates two vertices from the others,
as illustrated in Figure~\ref{fig-normaldiscs}.
Normal surfaces may be disconnected or empty.

\begin{figure}[htb]
    \centering
    \subfigure[Examples of normal discs]{%
        \label{fig-normaldiscs}
        \hspace{1.5cm}\includegraphics[scale=0.4]{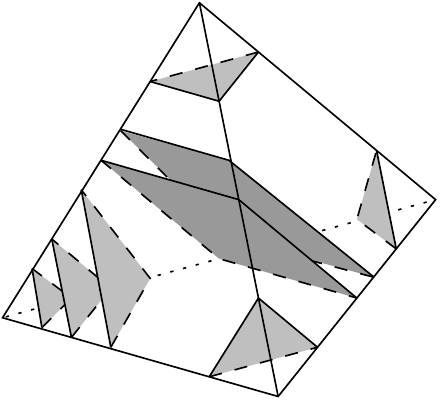}\hspace{1.5cm}}
    \hspace{1cm}
    \subfigure[The seven disc types in a tetrahedron]{%
        \label{fig-normaltypes} \includegraphics[scale=0.4]{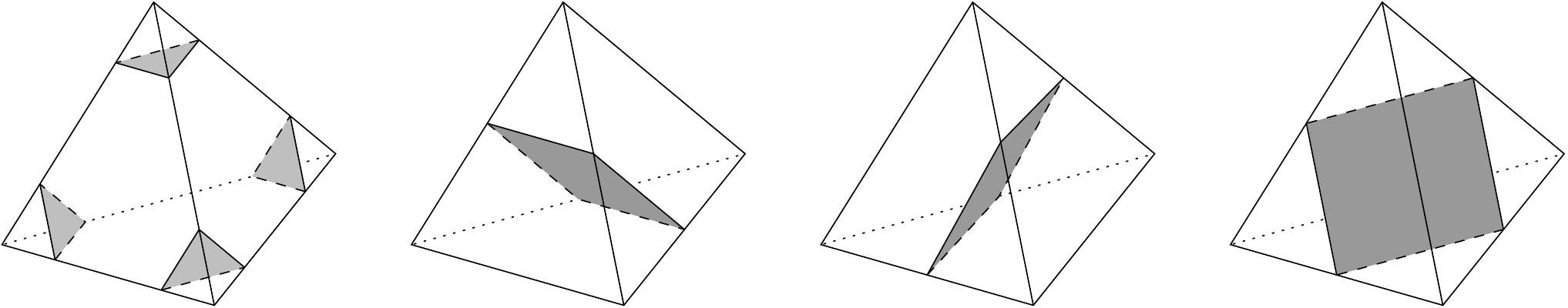}}
    \caption{Normal triangles and quadrilaterals}
    \label{fig-normal}
\end{figure}

There are seven \emph{types of disc} in each tetrahedron, defined by
which edges of the tetrahedron a normal disc meets.  These include
four triangle types and three quadrilateral types, all
illustrated in Figure~\ref{fig-normaltypes}.
The \emph{vector representation} of a normal surface $S$ is a vector
$\vrep{S} \in \Z^{7n}$, where the $7n$ elements of $\vrep{S}$ count the
number of discs of each type in each tetrahedron.
Given $\vrep{S}$, it is simple to reconstruct the surface $S$
(up to an isotopy that preserves the simplices of $\tri$).

For each normal surface $S$, $\vrep{S}$ satisfies the \emph{matching
equations}.  These are $3f$ linear homogeneous equations derived from $\tri$,
where $f$ is the number of non-boundary faces of $\tri$.  We express
these equations as $A \cdot \vrep{S} = 0$, where $A$ is a $3f \times 7n$
\emph{matching matrix}.  In essence, these equations ensure that
triangles and quadrilaterals in adjacent tetrahedra can be
joined together.
The matching matrix $A$ is sparse with small integer entries.
The \emph{normal surface solution cone} $\scone$ is the rational cone
$\scone = \{ \mathbf{x} \in \R^{7n}\,|\,
 A\mathbf{x}=0,\ \mathbf{x} \geq 0\}$,
which is a pointed cone with vertex at the origin.

No normal surface $S$ can have two different types of quadrilateral in the
same tetrahedron, since these would necessarily intersect (contradicting
the requirement that $S$ be properly embedded).
We say that a vector $\mathbf{x} \in \R^{7n}$ satisfies the
\emph{quadrilateral constraints} if, for each tetrahedron $\Delta$ of $\tri$,
at most one of the three coordinates of $\mathbf{x}$ that counts
quadrilaterals in $\Delta$ is non-zero.  Running through all $n$
tetrahedra, this gives us $n$ distinct
constraints of the form ``at most one of $x_i,x_j,x_k$ is non-zero''.
The quadrilateral constraints are
non-linear, and their solution set is non-convex.

A point $\mathbf{x} \in \R^{7n}$ is called \emph{admissible} if
it lies in the cone $\scone$ and satisfies the quadrilateral constraints.
By a theorem of Haken \cite{haken61-knot},
an integer vector $\mathbf{x} \in \Z^{7n}$ represents
a normal surface if and only if $\mathbf{x}$ is admissible.
The \emph{admissible region} of $\scone$ is the set of all admissible
points in $\scone$.

Let $\arbcone \subset \R^d$ be any pointed rational cone with
vertex at the origin,
and let $\intcone[d]{\arbcone}$ denote all integer points in $\arbcone$.
Then the \emph{Hilbert basis} of $\intcone[d]{\arbcone}$,
denoted $\hilb{\intcone[d]{\arbcone}}$ or just
$\hilb{\arbcone}$ for convenience, is the minimal set of integer points
that generates all of $\intcone[d]{\arbcone}$ under addition.
Equivalently, $\hilb{\arbcone}$ is the set of all
$\mathbf{x} \in \intcone[d]{\arbcone}$ for which,
if $\mathbf{x}=\mathbf{y}+\mathbf{z}$ with
$\mathbf{y},\mathbf{z}\in\intcone[d]{\arbcone}$,
then either $\mathbf{y}=0$ or $\mathbf{z}=0$.
Hilbert bases are finite and unique
\cite{hilbert88-basis,vandercorput31-hilbert},
and have many applications \cite{chubarov05-basis}.
They can be defined more generally, but for this paper the definition
above will suffice.

A \emph{fundamental normal surface} is a normal surface $S$ for which,
if $\vrep{S}=\vrep{T}+\vrep{U}$ for normal surfaces $T$ and $U$, then
either $\vrep{T}=0$ or $\vrep{U}=0$.
The fundamental normal surfaces are represented by the
admissible points in $\hilb{\scone}$, i.e.,
those points of $\hilb{\scone}$ that satisfy the quadrilateral constraints.

A \emph{vertex normal surface} is a normal surface $S$
for which $\vrep{S}$ lies on an extremal ray of $\scone$ and the
elements of $\vrep{S}$ have no common factor.\footnote{%
    There are several definitions of \emph{vertex normal surface}
    in the literature; others
    allow common factors \cite{burton09-convert,jaco03-0-efficiency},
    or use the double cover
    if $S$ is one-sided in $\tri$ \cite{burton12-ws,tollefson98-quadspace}.
    Our definition here follows that of Jaco and Oertel \cite{jaco84-haken}.}
It is clear that every vertex normal surface is also a fundamental
normal surface, but in general the converse is not true.

%
%

\section{The primal algorithm} \label{s-primal}

We begin this section by outlining the primal algorithm
described by Bruns, Ichim and Koch \cite{bruns10-normaliz,bruns01-closure}
for enumerating Hilbert bases in general.
Following this, we embed this into a larger
algorithm tailored for enumerating fundamental surfaces that exploits the
structure provided by normal surface theory.

The Bruns-Ichim-Koch algorithm comes in several forms
\cite{bruns10-normaliz,bruns01-closure}; here we describe the variant
most relevant to us.
Let $\arbcone \subset \R^d$
be a pointed rational cone with vertex at the origin.
The Bruns-Ichim-Koch algorithm takes as input the \emph{extremal rays}
of $\arbcone$, and computes the Hilbert basis $\hilb{\arbcone}$.
The following is an overview of the procedure;
full details can be found in \cite{bruns10-normaliz}.
\begin{enumerate}
    \item Use a variant of Fourier-Motzkin elimination to compute
    the supporting hyperplanes for $\arbcone$, and
    triangulate $\arbcone$ into simplicial subcones
    (cones whose extremal rays are linearly independent).

    \item For each simplicial subcone $\mathscr{S}$, build the
    semi-open parallelotope spanned by the smallest integer vector
    on each extremal ray of $\mathscr{S}$, as shown in
    Figure~\ref{fig-parallelotope}.  Specifically, if the extremal rays
    are defined by integer vectors
    $\mathbf{v}_1,\ldots,\mathbf{v}_k$, we build the parallelotope
    $\{\sum \lambda_i \mathbf{v}_i\,|\,0 \leq \lambda_i < 1\}$.

    \begin{figure}[htb]
        \centering
        \subfigure[Building a semi-open parallelotope]{
            \label{fig-parallelotope}
            \quad\includegraphics[scale=0.6]{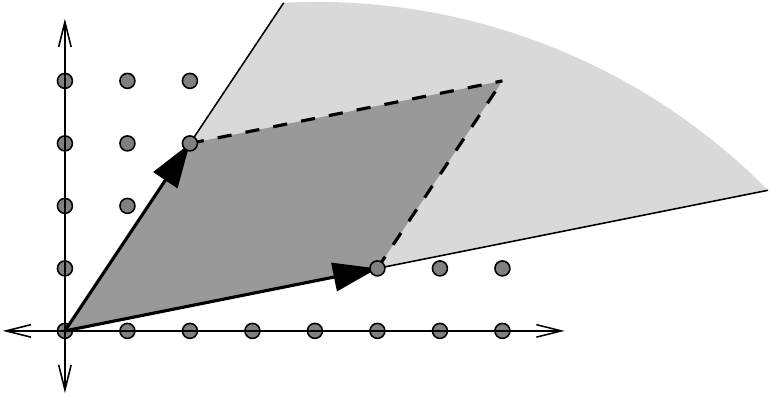}\quad}
        \hspace{1.5cm}
        \subfigure[All integer points in the parallelotope]{
            \label{fig-intpoints}
            \quad\qquad\includegraphics[scale=0.6]{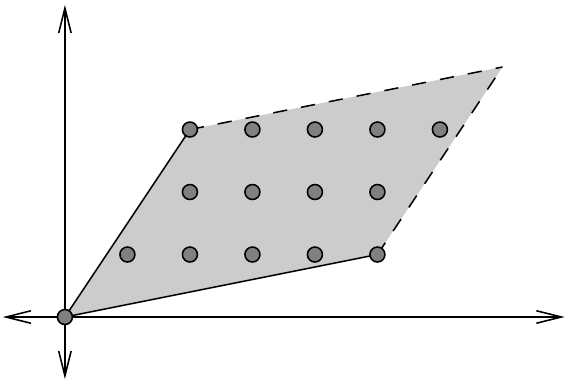}\quad\qquad}
        \caption{Building parallelotopes in the primal algorithm}
        \label{fig-primal}
    \end{figure}

    \item Identify all integer points within each semi-open parallelotope,
    as illustrated in Figure~\ref{fig-intpoints}, and combine these into a
    single list, which generates all of $\intcone[d]{\arbcone}$.
    Reduce both the intermediate and
    final lists by removing redundant generators
    (i.e., vectors that can be expressed as non-trivial sums of others in
    the lists).
    The final reduced list is the Hilbert basis $\hilb{\arbcone}$.
\end{enumerate}

In our setting,
we could run the Bruns-Ichim-Koch algorithm directly over
the normal surface solution cone $\scone \subset \R^{7n}$,
but this is typically infeasible given the high dimensionality of
$\scone$.  Moreover, it is wasteful:
our aim is to enumerate fundamental normal surfaces, and so we do not
need all of $\hilb{\scone}$, but just its \emph{admissible} points.

Our algorithm, instead of triangulating all of $\scone$,
triangulates only the (non-convex) admissible region of $\scone$.
We call a face $F \subseteq \scone$
\emph{admissible} if every point $\mathbf{x} \in F$ is admissible.
In particular, note that the vertex normal surfaces correspond
precisely to the admissible extremal rays of $\scone$.
A \emph{maximal admissible face} of $\scone$ is an admissible face that
is not a strict subface of another admissible face.

One can show that the admissible region of $\scone$ is
precisely the union of all maximal admissible faces
\cite{burton11-asymptotic,jaco84-haken}; in essence this decomposes the
\emph{non-convex} admissible region of $\scone$ into a union of
\emph{convex} subcones.  Note that
maximal admissible faces may have different dimensions, and if we
truncate the origin (which is common to all faces)
then the admissible region may even be disconnected.

Our strategy is to enumerate all maximal admissible faces of
$\scone$, and then use the Bruns-Ichim-Koch algorithm to triangulate
and compute the Hilbert basis for each.  This has the advantage that
maximal admissible faces are often significantly simpler than the
full cone $\scone$: they have lower dimension
\cite{burton11-asymptotic}, and are often generated by
very few admissible extremal rays \cite{burton10-complexity}.
This can significantly reduce the burden on the general Hilbert basis
enumeration procedure.
The overall algorithm is as follows:

\begin{algorithm}[Primal algorithm for fundamental normal surfaces]
        \label{a-primal}
    To enumerate all fundamental normal surfaces in a given triangulation:
    \begin{enumerate}
        \item Enumerate all vertex normal surfaces---that is,
        all admissible extremal rays of $\scone$---using the algorithms in
        \cite{burton09-convert} and \cite{burton13-tree}, which are
        optimised specifically for normal surface theory.
        Let $\mathcal{V}$ be the resulting set of admissible extremal rays.

        \item Using $\mathcal{V}$ as input,
        enumerate all maximal admissible faces
        of $\scone$ (see Algorithm~\ref{a-maxadm} below).

        \item For each maximal admissible face $F \subseteq \scone$,
        identify which rays in $\mathcal{V}$ are subfaces of $F$,
        and pass these rays as input to the Bruns-Ichim-Koch algorithm
        which will then enumerate $\hilb{F}$.
        The union $\bigcup \hilb{F}$ over all
        maximal admissible faces $F$ is
        is precisely the set of all vector representations of
        fundamental normal surfaces.
    \end{enumerate}
\end{algorithm}

\noindent
The correctness of the algorithm follows from the following observation:

\begin{lemma} \label{l-fundface}
    A vector $\mathbf{x} \in \R^{7n}$ represents a fundamental normal
    surface if and only if there is some maximal admissible face
    $F \subseteq \scone$ for which $\mathbf{x} \in \hilb{F}$.
\end{lemma}

\begin{proof}
    Let $\mathbf{x} = \vrep{S}$ where $S$ is a fundamental normal
    surface.  Since $\mathbf{x}$ is an admissible point of $\scone$,
    it must belong to some maximal admissible face $F \subseteq \scone$.
    If $\mathbf{x} \notin \hilb{F}$ then there
    are non-zero integer vectors $\mathbf{y},\mathbf{z} \in F$ for which
    $\mathbf{x}=\mathbf{y}+\mathbf{z}$.  Since $F$ is an admissible face
    of $\scone$ it follows that $\mathbf{y}=\vrep{T}$ and
    $\mathbf{z}=\vrep{U}$ for some
    normal surfaces $T$ and $U$, contradicting the fundamentality of $S$.

    Conversely, suppose that $\mathbf{x} \in \hilb{F}$ for some maximal
    admissible face $F \subseteq \scone$.  Then $\mathbf{x}$ is an
    admissible integer vector of $\scone$, and so $\mathbf{x}=\vrep{S}$
    for some normal surface $S$.  If $S$ is not fundamental then
    $\vrep{S}=\vrep{T}+\vrep{U}$ for normal surfaces $T$ and $U$ with
    $\vrep{T},\vrep{U} \neq 0$.
    Since they represent normal surfaces, both
    $\vrep{T},\vrep{U} \in \scone$, and so any face of the cone
    $\scone$ that contains $\vrep{S}$ must contain both summands
    $\vrep{T}$ and $\vrep{U}$ as well.
    In particular we have $\vrep{T},\vrep{U} \in F$, contradicting the
    assumption that $\mathbf{x} = \vrep{T}+\vrep{U} \in \hilb{F}$.
\end{proof}

\noindent
In the remainder of this section we expand upon
steps~2 and 3 of Algorithm~\ref{a-primal}.

We represent faces $F \subseteq \scone$
using \emph{zero sets} \cite{fukuda96-doubledesc}.
For each face $F \subseteq \scone$, the zero set of $F$ is a subset
of $\{1,\ldots,7n\}$ indicating
which coordinate positions are zero throughout $F$.
We denote this zero set by $Z(F)$, and define it formally as
$Z(F) = \{k\,|\,x_k = 0\ \mbox{for all}\ \mathbf{x} \in F\}$.

Because $\scone$ is of the form $\{ \mathbf{x} \in \R^{7n}\,|\,
 A\mathbf{x}=0,\ \mathbf{x} \geq 0\}$,
zero sets effectively encode the support hyperplanes for each face:
it is then clear that for any two faces $F,G \subseteq \scone$ we have
$Z(F) = Z(G)$ if and only if $F = G$, and
$Z(F) \subseteq Z(G)$ if and only if $G \subseteq F$.
Zero sets can be stored and manipulated efficiently in code using
bitmasks of length $7n$.

The enumeration of maximal admissible faces
makes use of admissible zero sets, which
correspond to admissible faces of $\scone$.  We call a set
$z \subseteq \{1,\ldots,7n\}$ \emph{admissible} if the corresponding
zero/non-zero coordinate pattern satisfies the quadrilateral constraints
in $\R^{7n}$;
that is, for each tetrahedron $\Delta$ of our triangulation, at most one
of the three coordinate positions that counts quadrilaterals in
$\Delta$ does not appear in $z$.

\begin{algorithm}[Maximal admissible face decomposition] \label{a-maxadm}
    Given the set $\mathcal{V}$ of all admissible extremal rays of $\scone$,
    the following algorithm enumerates all maximal admissible faces of $\scone$.
    \begin{enumerate}
        \item Build zero sets for all rays $R \in \mathcal{V}$ and
        store these in the set $\mathcal{S}_1$.
        Initialise an empty output set $\mathcal{M}$, which
        will eventually contain the zero sets for all
        maximal admissible faces of $\scone$.
        \item Inductively build sets $\mathcal{S}_2,\mathcal{S}_3,\ldots$
        as follows.  To build the set $\mathcal{S}_k$:
        \begin{enumerate}
            \item For each zero set $z \in \mathcal{S}_{k-1}$, find
            all $v \in \mathcal{S}_1$ for which
            $z \nsubseteq v$ and $z \cap v$ is admissible.
            For each such $v$, insert $z \cap v$ into $\mathcal{S}_k$.
            If no such $v$ exists, insert $z$ into the output set
            $\mathcal{M}$.
            \item Reduce $\mathcal{S}_k$ to its
            maximal elements by set inclusion.  That is, remove all
            $z \in \mathcal{S}_k$ for which there exists some
            strict superset $z' \in \mathcal{S}_k$.
            \item If $\mathcal{S}_k$ is empty, terminate the
            algorithm and return the output set $\mathcal{M}$.
        \end{enumerate}
    \end{enumerate}
\end{algorithm}

The key observation is that each set $\mathcal{S}_i$,
once constructed, contains precisely the zero sets for all
admissible faces of dimension $i$.  A full proof of
correctness and termination will be given shortly.

First, however, we make a simple observation regarding zero sets.
Recall from standard polytope theory that,
for any two faces $F$ and $G$ of a polyhedron $\mathscr{P}$,
the \emph{join} $F \vee G$ is the unique smallest-dimensional face
of $\mathscr{P}$ that contains both $F$ and $G$, and that for any face
$H \subseteq \mathscr{P}$ with $F,G \subseteq H$ we have
$F \vee G \subseteq H$ \cite{ziegler95}.

\begin{lemma} \label{l-join}
    For any two faces $F$ and $G$ of the normal surface solution cone $\scone$,
    we have $Z(F \vee G) = Z(F) \cap Z(G)$.
\end{lemma}

\begin{proof}
    Since $F \vee G \supseteq F,G$, any coordinate
    position that takes a non-zero value in either $F$ or $G$ also takes
    a non-zero value somewhere in $F \vee G$.  Therefore
    $Z(F \vee G) \subseteq Z(F) \cap Z(G)$.

    Let $H$ be the face of $\scone$ obtained by intersecting $\scone$
    with the supporting hyperplanes
    $\{x_i = 0\}$ for all $i \in Z(F) \cap Z(G)$.
    It is clear that $Z(F) \cap Z(G) \subseteq Z(H)$.
    Moreover, $F,G \subseteq H$ since every point in $F$ or $G$
    lies on all of the supporting hyperplanes listed above.
    Therefore $F \vee G \subseteq H$, and so
    $Z(F) \cap Z(G) \subseteq Z(H) \subseteq Z(F \vee G)$.

    We now have
    $Z(F \vee G) \subseteq Z(F) \cap Z(G) \subseteq Z(F \vee G)$,
    and so
    $Z(F \vee G) = Z(F) \cap Z(G)$.
\end{proof}

\noindent
We can now move on to the full proof of correctness and termination for
Algorithm~\ref{a-maxadm}.

\begin{proof}[Proof of correctness and termination]
    We first prove by induction that, once constructed, each set
    $\mathcal{S}_i$ contains precisely the zero sets for all admissible
    faces of $\scone$ of dimension $i$.

    This is clearly true for $i=1$, since the 1-dimensional faces of
    $\scone$ are precisely the extremal rays, and we use all
    admissible extremal rays of $\scone$
    to fill $\mathcal{S}_1$ in step~1 of the algorithm.

    Consider now some $i > 1$.  We assume our inductive hypothesis
    for both $\mathcal{S}_{i-1}$ and $\mathcal{S}_1$, and prove it for
    $\mathcal{S}_i$ in three stages:

    \begin{enumerate}[(i)]
        \item \emph{Every $y \in \mathcal{S}_i$ is the zero set
        $Z(F)$ for some admissible face $F \subseteq \scone$
        of dimension $\geq i$.}

        Suppose that $y$ is constructed in step~2(a) of the algorithm as
        $y = z \cap v$, where $z \in \mathcal{S}_{i-1}$
        and $v \in \mathcal{S}_1$.  By the inductive hypothesis
        we have $z=Z(G)$ and $v=z(R)$ for some admissible faces
        $G,R \subseteq \scone$ with
        $\dim(G) = i-1$ and $\dim(R)=1$.

        By Lemma~\ref{l-join} we have $y = Z(G \vee R)$.
        Because step~2(a) of the
        algorithm only constructs admissible zero sets $y = z \cap v$,
        it follows that $G \vee R$ must be an admissible face.
        Because step~2(a) only considers pairs for which $z \nsubseteq v$,
        it follows that $R \nsubseteq G$ and therefore
        $G \vee R$ has strictly higher
        dimension than $G$; that is, $\dim(G \vee R) \geq i$.

        \item \emph{For every admissible $i$-dimensional face
        $F \subseteq \scone$, we have $Z(F) \in \mathcal{S}_i$.}

        Let $G \subseteq F$ be any $(i-1)$-dimensional subface of $F$,
        and let $R \subseteq F$ be any extremal ray of $F$ not contained in
        $G$.
        Since $F \supseteq G,R$ we have $F \supseteq G \vee R$,
        and since $i-1 = \dim(G) < \dim(G \vee R) \leq \dim(F) = i$
        it follows that $\dim(F) = \dim(G \vee R)$, and hence $F = G \vee R$.
        By Lemma~\ref{l-join} we then have $Z(F) = Z(G) \cap Z(R)$.

        Since $F$ is admissible, so are $G$ and $R$, and it follows
        from the inductive hypothesis that $Z(G) \in \mathcal{S}_{i-1}$
        and $Z(R) \in \mathcal{S}_1$.
        Because $R \nsubseteq G$ we have $Z(G) \nsubseteq Z(R)$,
        and we already know from $F$ that $Z(F) = Z(G) \cap Z(R)$ is admissible.
        Therefore the set
        $Z(F) = Z(G) \cap Z(R)$ is inserted into $\mathcal{S}_i$
        in step~2(a) of the algorithm.

        Furthermore, $Z(F)$ is not removed from $\mathcal{S}_i$
        in step~2(b) of the algorithm.
        This is because, by stage~(i) above, any strict superset
        $z' \supset Z(F)$ that appears in $\mathcal{S}_i$ must correspond to
        a strict subface $F' \subset F$ of dimension $\geq i$.  This is
        impossible, since $F$ itself has dimension $i$.

        \item \emph{For every admissible $j$-dimensional face
        $F \subseteq \scone$ with $j > i$,
        we have $Z(F) \notin \mathcal{S}_i$.}

        For any such $F$, let $G \subseteq F$ be any $i$-dimensional
        subface of $F$.  Because $F$ is admissible, $G$ is also,
        and by stage~(ii) above it follows that
        $Z(G) \in \mathcal{S}_i$.  Since $G$ is a strict subface of $F$,
        $Z(G)$ is a strict superset of $Z(F)$, and so even if $Z(F)$ is
        constructed in step~2(a) of the algorithm, it will be
        subsequently removed in step~2(b).
    \end{enumerate}

    \noindent
    Together these three stages establish our inductive claim.

    From here it is simple to see that Algorithm~\ref{a-maxadm} terminates:
    since the cone $\scone$ has finitely many faces, there are only
    finitely many non-empty sets $\mathcal{S}_k$.
    Note also that Algorithm~\ref{a-maxadm} does not terminate until all
    non-empty sets $\mathcal{S}_k$ have been constructed, since
    if $\mathcal{S}_k$ is empty then there are no admissible $k$-faces,
    which means there can be no admissible $i$-faces for any $i \geq k$.

    To finish, we show that the output of Algorithm~\ref{a-maxadm} is correct.
    Let $F \subseteq \scone$ be any maximal admissible face of dimension
    $i$; by our induction above we have $Z(F) \subseteq \mathcal{S}_i$.
    Suppose there is some $v \in \mathcal{S}_1$
    for which $Z(F) \nsubseteq v$ and for which $Z(F) \cap v$ is admissible.
    Following the argument in stage~(i) above,
    $Z(F) \cap v$ must be the zero set for some admissible face of dimension
    $\geq i+1$.  Moreover, this must be a strict superface of $F$,
    which contradicts the maximality of $F$.
    Therefore there can be no such $v \in \mathcal{S}_1$, and so we include
    $Z(F)$ in the output set $\mathcal{M}$ as required.

    Conversely, consider any output $z \in \mathcal{M}$.  We must have
    $z \in \mathcal{S}_i$ for some $i$, and so $z = Z(G)$ for some
    admissible $i$-dimensional face $G \subseteq \scone$.
    If $G$ is not maximal, there must be some admissible
    $(i+1)$-dimensional superface $F \supset G$,
    and some admissible extremal ray $R \subseteq F$ not contained in $G$.
    Following the argument in stage~(ii) above,
    for the pair $Z(G) \in \mathcal{S}_i$ and $Z(R) \in \mathcal{S}_1$
    we have $Z(G) \nsubseteq Z(R)$ and $Z(G) \cap Z(R)$ is admissible,
    which means we would not insert $z$ into the output set
    $\mathcal{M}$ in step~2(a) of the algorithm.
    Therefore $z = Z(G)$ for some maximal admissible face $G \subseteq \scone$.
\end{proof}

In Algorithm~\ref{a-maxadm},
note that each zero set in step~1 can be built by examining any non-zero
representative of the corresponding ray $R$.
In step~2(a), a set might be constructed from many different
$(z,v)$ pairs, and so it is important when inserting $z \cap v$
into $\mathcal{S}_k$ to avoid duplicates.
Each $\mathcal{S}_k$ is only required for the following inductive step
(building $\mathcal{S}_{k+1}$), and can then be discarded.

Returning to Algorithm~\ref{a-primal},
in step~3 we can use zero sets to quickly identify which rays
are subfaces of $F$:
a ray $R \in \mathcal{V}$ is a subface of $F$ if and only if
$Z(R) \supseteq Z(F)$.
The Bruns-Ichim-Koch algorithm does enumerate $\hilb{F}$
as described, since each
extremal ray of $F$ is admissible, and so the rays in $\mathcal{V}$
that are subfaces of $F$ (which we give Bruns-Ichim-Koch as input)
are indeed the extremal rays of $F$.

%
%

\section{The dual algorithm} \label{s-dual}

We now outline the dual algorithm of
Bruns and Ichim \cite{bruns10-normaliz} for general Hilbert basis
enumeration, which is based on earlier work of
Pottier \cite{pottier96-euclidean}, and then show how this can be seamlessly
modified to yield a dual algorithm for enumerating fundamental normal
surfaces.

The Bruns-Ichim-Pottier dual algorithm has several variants,
and we describe the variant most relevant to us.
Consider the cone
$\mathscr{P} = \{ \mathbf{x} \in \R^d\,|\,A\mathbf{x}=0,\ \mathbf{x} \geq 0\}$,
where $A$ is some rational $m \times d$ matrix.
The dual algorithm takes as input the matrix $A$, and computes the
Hilbert basis $\hilb{\mathscr{P}}$.

It constructs a series of cones
$\mathscr{P}^{(0)},\ldots,\mathscr{P}^{(m)}$,
beginning with the non-negative orthant $\mathscr{P}^{(0)}=\R^d_{\geq 0}$,
and finishing with the target cone $\mathscr{P}^{(m)}=\mathscr{P}$.
In general, $\mathscr{P}^{(k)}$ is defined as for $\mathscr{P}$ above
but using only the first $k$ rows of the matrix $A$.
At each stage, the dual algorithm uses the previous basis
$\hilb{\mathscr{P}^{(k-1)}}$ to inductively compute the next basis
$\hilb{\mathscr{P}^{(k)}}$.
In this sense, the dual algorithm generalises the
double description method for enumerating extremal rays
\cite{fukuda96-doubledesc,motzkin53-dd}.

The key inductive step of the dual algorithm is the following
\cite{bruns10-normaliz}.

\begin{algorithm}[Bruns-Ichim-Pottier inductive step] \label{a-inductive}
    Let $\mathscr{P} \subset \R^d$ be a pointed rational cone
    with vertex at the origin,
    let $\mathbf{h} \in \R^d$ be a rational vector, and define the cones
    $\mathscr{P}_+ =
    \{\mathbf{x} \in \mathscr{P}\,|\,\mathbf{h}\cdot\mathbf{x} \geq 0\}$,
    $\mathscr{P}_- =
    \{\mathbf{x} \in \mathscr{P}\,|\,\mathbf{h}\cdot\mathbf{x} \leq 0\}$,
    and
    $\mathscr{P}_0 =
    \{\mathbf{x} \in \mathscr{P}\,|\,\mathbf{h}\cdot\mathbf{x}=0\}$.
    Then, given the input basis $B = \hilb{\mathscr{P}}$,
    the following algorithm computes the Hilbert bases
    $\hilb{\mathscr{P}_+}$, $\hilb{\mathscr{P}_-}$
    and $\hilb{\mathscr{P}_0}$.
    \begin{enumerate}
        \item Initialise candidate bases
        $B_+ = \{\mathbf{x} \in B\,|\,\mathbf{h}\cdot\mathbf{x}\geq 0\}$
        and
        $B_- = \{\mathbf{x} \in B\,|\,\mathbf{h}\cdot\mathbf{x}\leq 0\}$.

        \item Expand these candidate bases:
        for all pairs $\mathbf{x} \in B_+$ and $\mathbf{y} \in B_-$
        with $\mathbf{h}\cdot\mathbf{x} > 0 > \mathbf{h}\cdot\mathbf{y}$,
        insert the sum $\mathbf{x}+\mathbf{y}$ into $B_+$ and/or $B_-$
        according to whether
        $\mathbf{h}\cdot(\mathbf{x} + \mathbf{y})\geq 0$ and/or
        $\mathbf{h}\cdot(\mathbf{x} + \mathbf{y})\leq 0$.

        \item Reduce the candidate bases:
        remove all $\mathbf{b} \in B_+$ for which there exists some
        different $\mathbf{b}' \in B_+$ with
        $\mathbf{b}-\mathbf{b}' \in \mathscr{P}_+$.
        Reduce $B_-$ in a similar fashion.

        \item If $B_+$ and $B_-$ are both unchanged after the most
        recent round of expansion and reduction (steps 2 and 3),
        terminate with
        $\hilb{\mathscr{P}_+}=B_+$,
        $\hilb{\mathscr{P}_-}=B_-$, and
        $\hilb{\mathscr{P}_0}=B_+ \cap B_-$.
        Otherwise return to step~2 for a new round of expansion and
        reduction.
    \end{enumerate}
\end{algorithm}

In essence, step~2 expands the set of integer points that are generated under
addition by each individual set $B_+$ and $B_-$, and step~3 removes
redundant vectors from each generating set.
See \cite{bruns10-normaliz} for full proofs of
correctness and termination, both of which are non-trivial results.

There are many ways to optimise this algorithm.  In practice, one
can split $B_+$ and $B_-$ into three sets according to whether
$\mathbf{h}\cdot\mathbf{x}>0$,
$\mathbf{h}\cdot\mathbf{x}<0$ or
$\mathbf{h}\cdot\mathbf{x}=0$, to avoid duplication along
the common hyperplane $\mathbf{h}\cdot\mathbf{x}=0$.  The expansion and
reduction steps can be partially merged, so that we first
test each new sum $\mathbf{x}+\mathbf{y}$ for reduction before inserting
it into $B_+$ and/or $B_-$.
Several other optimisations are possible:
see \cite[Remark~16]{bruns10-normaliz}, and the
``darwinistic reduction'' in \cite[Section~3]{bruns10-normaliz}.

As with the primal method, running the dual algorithm of
Bruns, Ichim and Pottier
directly over the normal surface solution cone $\scone \subset \R^{7n}$
can be extremely slow.  Like the double description method for
vertex enumeration, one significant
problem is \emph{combinatorial explosion}:
even if the final Hilbert basis $\hilb{\mathscr{P}}$ is small,
the intermediate bases $\hilb{\mathscr{P}^{(k)}}$ can be extremely
large.  The high dimension of the normal surface solution cone $\scone$
exacerbates this problem.

Our solution, as with the primal algorithm, is to restrict our attention
to just the admissible region of $\scone$.  This time we do not decompose the
admissible region into maximal admissible faces; instead we can embed the
quadrilateral constraints seamlessly into the algorithm itself.  In this way,
a single run through the dual algorithm can compute all admissible
Hilbert basis elements for $\scone$.

The idea is simple:
when expanding
candidate bases, ignore sums $\mathbf{x}+\mathbf{y}$ that do not
satisfy the quadrilateral constraints.
This mirrors Letscher's filter for
the double description method \cite{burton10-dd}.

\begin{algorithm}[Dual algorithm for fundamental normal surfaces] \label{a-dual}
    To enumerate all fundamental normal surfaces in a triangulation
    with the $3f \times 7n$ matching matrix $A$:
    \begin{enumerate}
        \item Reorder the rows of $A$ using a good heuristic (as discussed
        further below).
        \item Initialise the candidate basis $B^{(0)}$ with all unit vectors
        in $\R^{7n}$.
        \item Inductively construct candidate bases
        $B^{(1)},\ldots,B^{(3f)}$ as follows.  For each $i=1,\ldots,3f$,
        run a modified Algorithm~\ref{a-inductive} with input basis $B^{(i-1)}$,
        using the vector $\mathbf{h} \in \R^{7n}$ defined by the
        $i$th row of $A$.  The specific modifications to
        Algorithm~\ref{a-inductive} are:
        \begin{itemize}
            \item
            In step~2, only consider pairs
            $\mathbf{x},\mathbf{y}$ for which
            $\mathbf{x}+\mathbf{y}$ satisfies the quadrilateral constraints.
            \item
            In step~3, instead of testing for
            $\mathbf{b}-\mathbf{b}' \in \mathscr{P}_+$,
            test for
            $\mathbf{b}-\mathbf{b}' \geq 0$ and
            $\mathbf{h}\cdot(\mathbf{b}-\mathbf{b}') \geq 0$.
            Instead of testing for
            $\mathbf{b}-\mathbf{b}' \in \mathscr{P}_-$,
            test for
            $\mathbf{b}-\mathbf{b}' \geq 0$ and
            $\mathbf{h}\cdot(\mathbf{b}-\mathbf{b}') \leq 0$.
        \end{itemize}
        When this modified Algorithm~\ref{a-inductive} terminates with output
        $B_+$ and $B_-$, set $B^{(i)}=B_+ \cap B_-$.
        \item The final set $B^{(3f)}$ is then the set of all vector
        representations of fundamental normal surfaces.
    \end{enumerate}
\end{algorithm}

Although we use the Bruns-Ichim-Pottier inductive process
in step~3 of our algorithm, we never identify any explicit cone
$\mathscr{P} \subset R^{7n}$ for which the input set $B^{(i-1)}$
is the Hilbert basis (this is because we only ever track admissible
basis elements).  This is why we modify step~3 of
Algorithm~\ref{a-inductive}: we must remove any reference to
the unknown cone $\mathscr{P}$.

The key invariant in this dual algorithm is that each candidate basis
$B^{(i)}$ contains precisely those Hilbert basis elements
for the cone defined by the first $i$ rows of the matching matrix $A$
that satisfy the quadrilateral constraints.
See the appendix for full proofs of correctness and termination.

The order in which we process the rows of $A$ can have a significant
effect on performance, by affecting the
severity of the combinatorial explosion
\cite{bruns10-normaliz,fukuda96-doubledesc}.
In step~1 of Algorithm~\ref{a-dual}, we reorder the rows using
\emph{position vectors} \cite{burton10-dd}, which essentially favours rows that
begin with long strings of zeroes.  Position vectors help optimise our
quadrilateral constraint filter, and have been found to
perform well as a sorting heuristic
for the related double description method;
see \cite{burton10-dd} for details.

We can prove correctness and termination for Algorithm~\ref{a-dual} as follows.

\begin{proof}[Proof of correctness and termination]
    Let $A^{(k)}$ denote the $k \times 7n$ submatrix of $A$ obtained by
    taking the first $k$ rows (after the reordering in step~1 of the
    algorithm), and let
    $\mathscr{P}^{(k)} =
    \{\mathbf{x} \in \R^{7n}\,|\,A^{(k)}\mathbf{x}=0,\ \mathbf{x}\geq 0\}$.
    We prove by induction that each candidate basis $B^{(k)}$ constructed in
    Algorithm~\ref{a-dual} consists of all elements of
    $\hilb{\mathscr{P}^{(k)}}$ that satisfy the quadrilateral constraints.

    The initial case is simple: $\mathscr{P}^{(0)}$ is just the non-negative
    orthant $\R^{7n}_{\geq 0}$, whose Hilbert basis consists of the unit
    vectors in $\R^{7n}$, all of which satisfy the quadrilateral
    constraints.  This matches the
    initialisation of $B^{(0)}$ in step~2 of the algorithm.

    Assume now that $B^{(i-1)}$ contains all elements of
    $\hilb{\mathscr{P}^{(i-1)}}$ that satisfy the quadrilateral
    constraints, and consider the construction of
    $B^{(i)}$ in step~3 of Algorithm~\ref{a-dual}.  This involves running
    a modified Algorithm~\ref{a-inductive}
    (the Bruns-Ichim-Pottier inductive step), with input basis
    $B^{(i-1)}$ and with the vector $\mathbf{h} \in \R^{7n}$ defined by
    the $i$th row of the (sorted) matrix $A$.

    Let
    $\mathscr{P}_+ =
    \{\mathbf{x}\in\mathscr{P}^{(i-1)}\,|\,\mathbf{h}\cdot\mathbf{x}\geq 0\}$,
    $\mathscr{P}_- =
    \{\mathbf{x}\in\mathscr{P}^{(i-1)}\,|\,\mathbf{h}\cdot\mathbf{x}\leq 0\}$,
    and
    $\mathscr{P}_0 = \mathscr{P}_+ \cap \mathscr{P}_- =
    \{\mathbf{x}\in\mathscr{P}^{(i-1)}\,|\,\allowbreak
    \mathbf{h}\cdot\mathbf{x}=0\}$.
    We claim that the modified Algorithm~\ref{a-inductive}
    terminates with output sets $B_+$ and $B_-$ containing those
    elements of $\hilb{\mathscr{P}_+}$ and $\hilb{\mathscr{P}_-}$
    respectively that satisfy the quadrilateral constraints.
    If this is true, then setting $B^{(i)} = B_+ \cap B_-$ will fill
    $B^{(i)}$ with those elements
    of $\hilb{\mathscr{P}_0} = \hilb{\mathscr{P}^{(i)}}$ that satisfy
    the quadrilateral constraints, as required.
    This will complete our induction, and
    the final output $B^{(3f)}$ will then consist of all admissible elements of
    $\hilb{\mathscr{P}^{(3f)}} = \hilb{\scone}$; that is,
    all vector representations of fundamental normal surfaces.

    It remains to prove our claim above, i.e., that
    the modified Algorithm~\ref{a-inductive}
    terminates with output sets $B_+$ and $B_-$ containing those
    elements of $\hilb{\mathscr{P}_+}$ and $\hilb{\mathscr{P}_-}$
    respectively that satisfy the quadrilateral constraints.

    Consider running the modified
    Algorithm~\ref{a-inductive} and the original
    Algorithm~\ref{a-inductive} side-by-side, where in the original
    algorithm we use the full basis $\hilb{\mathscr{P}^{(i-1)}}$ as input.
    Let $B_+^m$ and $B_-^m$ denote the working sets $B_+$ and $B_-$ in our
    modified
    algorithm, and let $B_+^o$ and $B_-^o$ denote the working sets
    $B_+$ and $B_-$ in
    the original algorithm.  From Bruns and Ichim
    \cite{bruns10-normaliz}, we know that the original algorithm terminates
    with $B_+^o = \hilb{\mathscr{P}_+}$ and
    $B_-^o = \hilb{\mathscr{P}_-}$.

    We claim that, for as long as both algorithms are running
    side-by-side, the working
    sets $B_+^m$ and $B_-^m$ contain precisely those elements of
    $B_+^o$ and $B_-^o$ respectively that satisfy the quadrilateral
    constraints.  We show this again using induction:
    \begin{itemize}
        \item It is true when the algorithms begin, since we assume from
        our ``outer induction'' on $B^{(i)}$ that the input basis
        $B^{(i-1)}$ for the modified algorithm contains precisely those
        elements that satisfy the quadrilateral constraints
        from the input basis $\hilb{\mathscr{P}^{(i-1)}}$
        for the original algorithm.

        \item We now show that, if at some point $B_+^m$ and $B_-^m$ contain the
        elements of $B_+^o$ and $B_-^o$ that satisfy the quadrilateral
        constraints, then this remains true after a single
        round of expansion (step~2 of Algorithm~\ref{a-inductive}).

        It is clear from our modification in step~3 of Algorithm~\ref{a-dual}
        that we do not insert any new sums
        $\mathbf{x}+\mathbf{y}$ into $B_+^m$ or $B_-^m$ that
        \emph{break} the quadrilateral constraints.
        Moreover, in the original algorithm,
        any sum $\mathbf{x}+\mathbf{y}$ that \emph{does} satisfy the
        quadrilateral constraints must have
        summands $\mathbf{x} \in B_+^o$ and $\mathbf{y} \in B_-^o$
        that satisfy them also
        (since $\mathbf{x},\mathbf{y} \geq 0$); therefore
        $\mathbf{x} \in B_+^m$ and $\mathbf{y} \in B_-^m$,
        and the sum $\mathbf{x}+\mathbf{y}$
        is also considered in the modified algorithm.

        \item Likewise, we can show that if at some point $B_+^m$ and $B_-^m$
        contain the elements of $B_+^o$ and $B_-^o$ that satisfy the
        quadrilateral constraints, then this remains true
        after a single round of reduction
        (step~3 of Algorithm~\ref{a-inductive}).

        First, we note that our modification to the reduction step is
        sound: for any $\mathbf{b},\mathbf{b}' \in B_+^m$, we have
        $\mathbf{b}-\mathbf{b}' \in \mathscr{P}_+$ if and only if
        $\mathbf{b}-\mathbf{b}' \geq 0$ and
        $\mathbf{h}\cdot(\mathbf{b}-\mathbf{b}') \geq 0$,
        since we already have
        $A^{(i-1)}\mathbf{b}=A^{(i-1)}\mathbf{b}'=0$
        from our outer induction.
        Next, we observe that a vector $\mathbf{b} \in B_+^m$
        is removed in step~3 of the modified Algorithm~\ref{a-inductive}
        if and only if it is removed
        in the original algorithm: this is because
        $\mathbf{b}$ satisfies the quadrilateral constraints, and so any
        $\mathbf{b}'$ for which
        $\mathbf{b}-\mathbf{b}' \geq 0$ must satisfy them also,
        and therefore $\mathbf{b}' \in B_+^m$ if and only if
        $\mathbf{b}' \in B_+^o$.
        Similar arguments apply to $\mathscr{P}_-$ and $B_-^m$.
    \end{itemize}

    By induction, $B_+^m$ and $B_-^m$ contain precisely those elements of
    $B_+^o$ and $B_-^o$ that satisfy the quadrilateral constraints
    for as long as both algorithms are running.

    It follows that the modified algorithm terminates no later
    than the original algorithm, since if $B_+^o$ and $B_-^o$ are unchanged
    under some round of expansion and reduction, then their elements
    $B_+^m$ and $B_-^m$ that satisfy the quadrilateral constraints
    are unchanged also.
    The modified algorithm might terminate \emph{earlier}, but this will
    not change the final output: if $B_+^m$ and $B_-^m$ are
    unchanged under some earlier round of expansion and reduction, then
    running any additional rounds will leave them unchanged.
    
    Either way, we see that
    the final output sets $B_+^m$ and $B_-^m$ from the modified
    Algorithm~\ref{a-inductive}
    contain precisely those elements of the final output sets
    $B_+^o = \hilb{\mathscr{P}_+}$ and $B_-^o = \hilb{\mathscr{P}_-}$
    that satisfy the quadrilateral
    constraints, thus establishing our original claim.
\end{proof}

%
%

\section{Experimentation} \label{s-perf}

For our experiments, we implement both the primal and dual algorithms
directly in the software package
{\regina} \cite{burton04-regina,regina}.
The primal algorithm uses the tree traversal method \cite{burton13-tree}
for step~1 (enumerating vertex surfaces), and calls
{\normaliz}~2.10.1 \cite{bruns10-normaliz,bruns01-closure}
for step~3 (the Bruns-Ichim-Koch algorithm).
The dual algorithm incorporates optimisations
described by Bruns and Ichim \cite{bruns10-normaliz}.
All implementations use arbitrary-precision integer arithmetic.

Our experimental data set is the closed hyperbolic census of Hodgson and
Weeks \cite{hodgson94-closedhypcensus}, which consists of 10,986
3-manifold triangulations with sizes from $n=9$ to $n=35$ tetrahedra.
This census was chosen because (i)~it
contains ``real-world'' triangulations with properties and structures
that one might encounter in a typical topological algorithm; and
(ii)~the triangulations are large enough to make the problems difficult
and the results mathematically interesting.

The triangulations were sorted by increasing $n$, and then each
algorithm was run over the census on a cluster, processing triangulations in
parallel until a fixed walltime limit was reached.
For the primal algorithm, the limit was 4~days walltime with 15 slave
processes (each of which could process one triangulation at a time).
For the dual algorithm it became clear early on that the results would
be severely limited (as discussed below),
and so it was rerun with 4~days walltime and
31 slave processes to increase the number of triangulations for which
the two algorithms could be compared.
All processes ran on 2.93GHz Intel Xeon X5570 CPUs.

\begin{figure}[tb]
    \centering
    \includegraphics[scale=0.5]{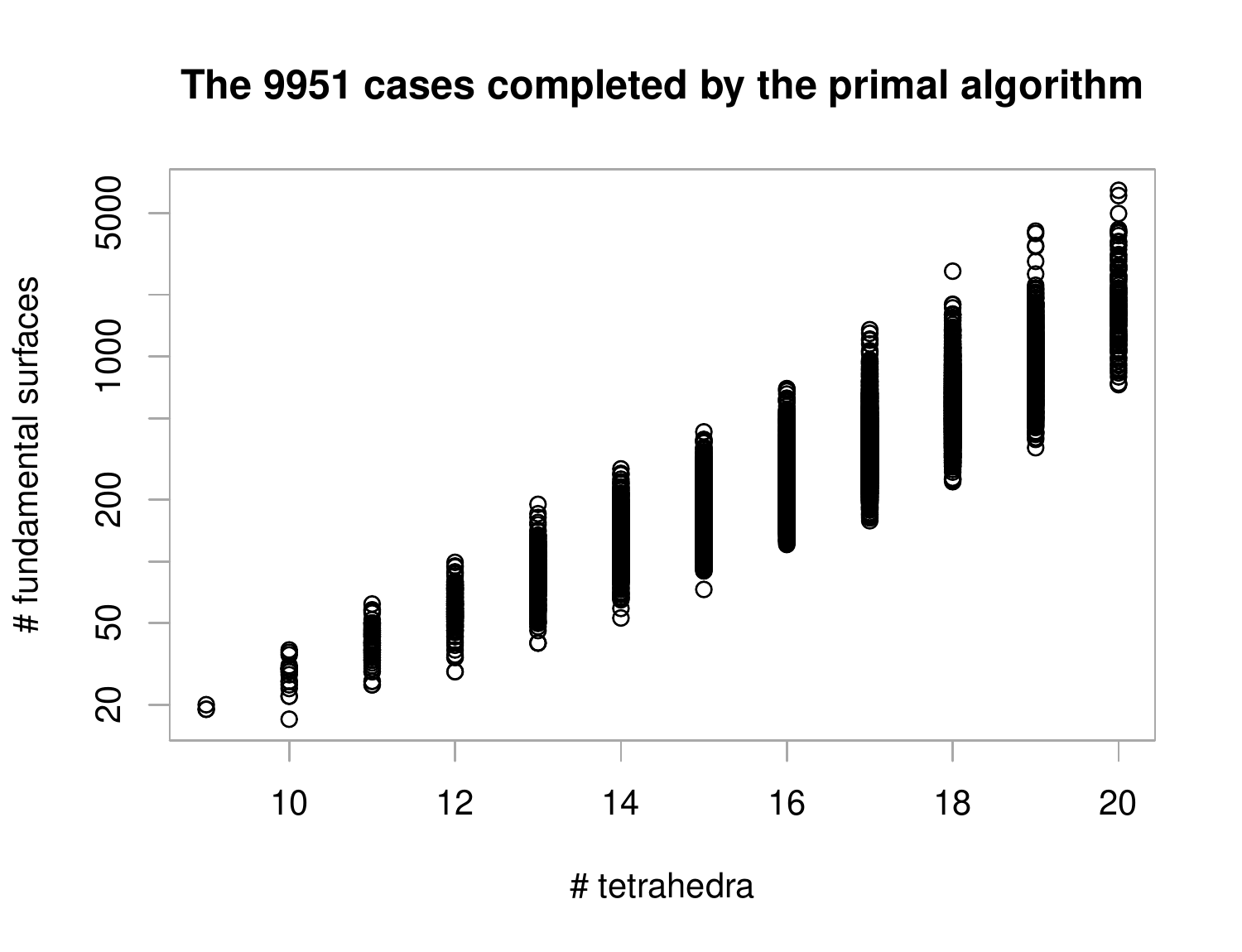}%
    \includegraphics[scale=0.5]{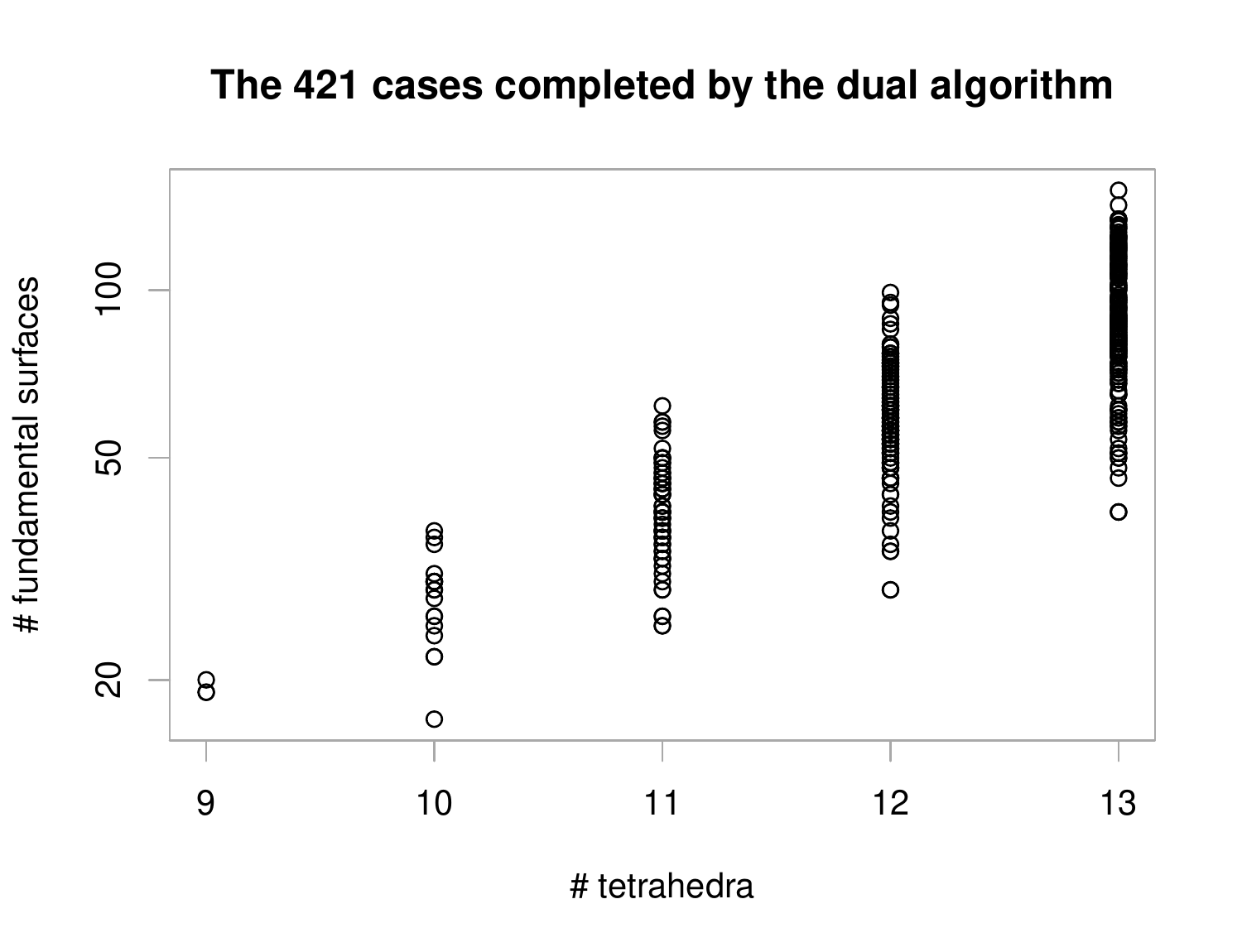}
    \caption{Input and output sizes for test cases that were completed}
    \label{fig-cases}
\end{figure}

The primal algorithm completed 9951 of the 10986 test cases, and the
dual algorithm completed just 421.  Figure~\ref{fig-cases}
summarises the input and output sizes for these cases, with
output size on a log scale.
Figure~\ref{fig-comparison} compares running times for the
421 cases that were completed by
both algorithms (there were no cases that the dual algorithm completed
but the primal algorithm did not).
Each point on the plot represents an individual
triangulation, and here both axes use log scales.

\begin{figure}[tb]
    \centering
    \includegraphics[scale=0.55]{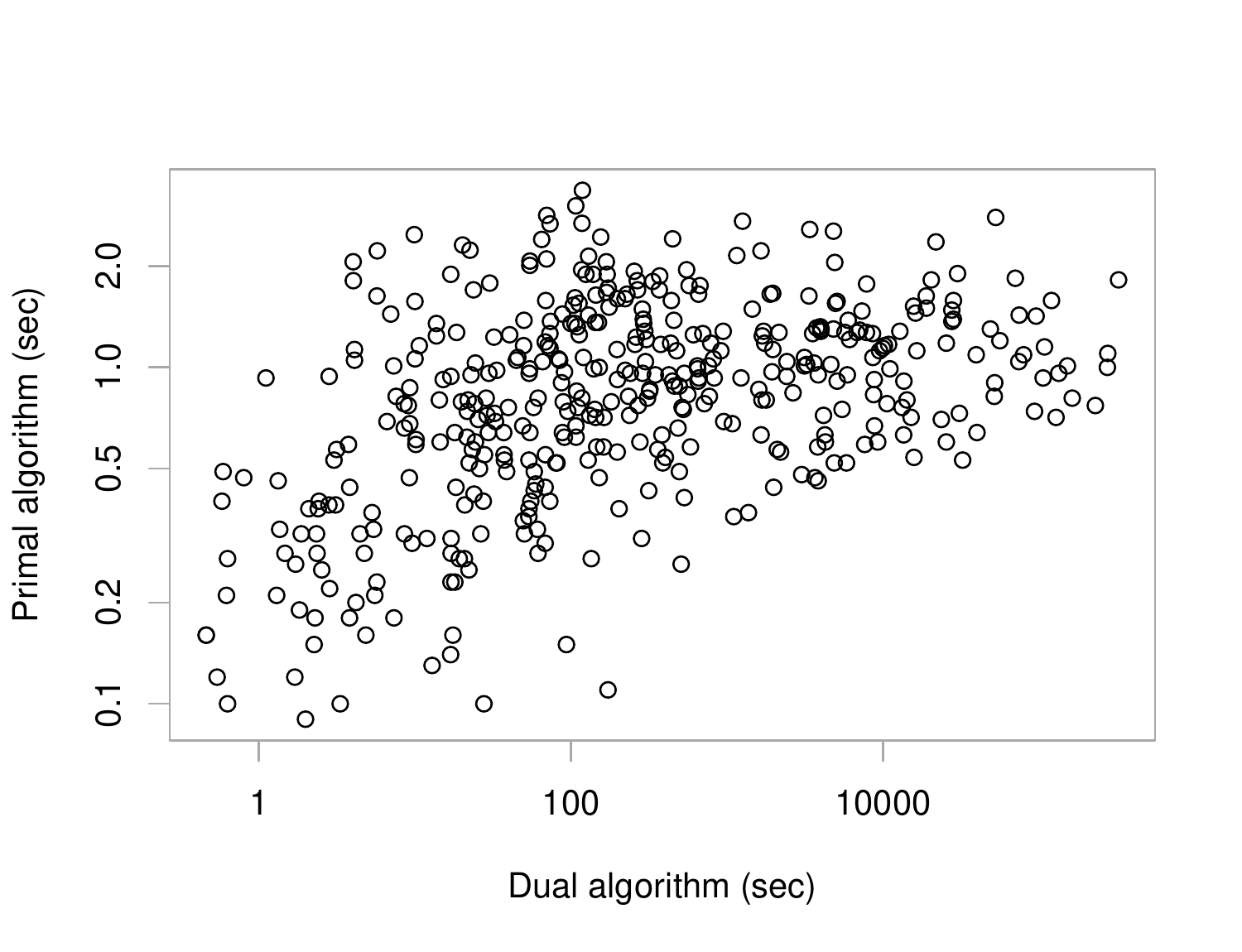}
    \caption{Comparison of running times for the cases solved by
    both algorithms}
    \label{fig-comparison}
\end{figure}

It is clear from the figure that the primal algorithm runs significantly
faster.  This contrasts with the enumeration of \emph{vertex} normal surfaces,
where the corresponding double description is fast for problems of this
size (but becomes inferior when the problems become larger)
\cite{burton13-tree}.
The primal method is also more consistent in its performance:
for $n=13$ (the largest amongst these 421 cases),
the primal algorithm ranges from
$0.44$ to $3.36$ seconds (spanning a single order of magnitude),
whereas the dual algorithm ranges from
$4.02$ to $322\,544$ seconds (spanning five orders of magnitude).


We now drill more deeply into the primal algorithm, to find where its
bottlenecks lie.  For this we use all 9951 cases that it completed,
so that we can study how the algorithm behaves for larger and more
difficult problems.
Figure~\ref{fig-perf-primal} shows how the running time is divided
amongst the three main steps of (i)~enumerating vertex surfaces;
(ii)~constructing maximal admissible faces; and (iii)~running the
Bruns-Ichim-Koch algorithm over each.  For each plot, the horizontal
axis measures total running time (on a log scale), as an indicator of
how difficult overall each test case was found to be.

\begin{figure}[tb]
    \centering
    \includegraphics[scale=0.5]{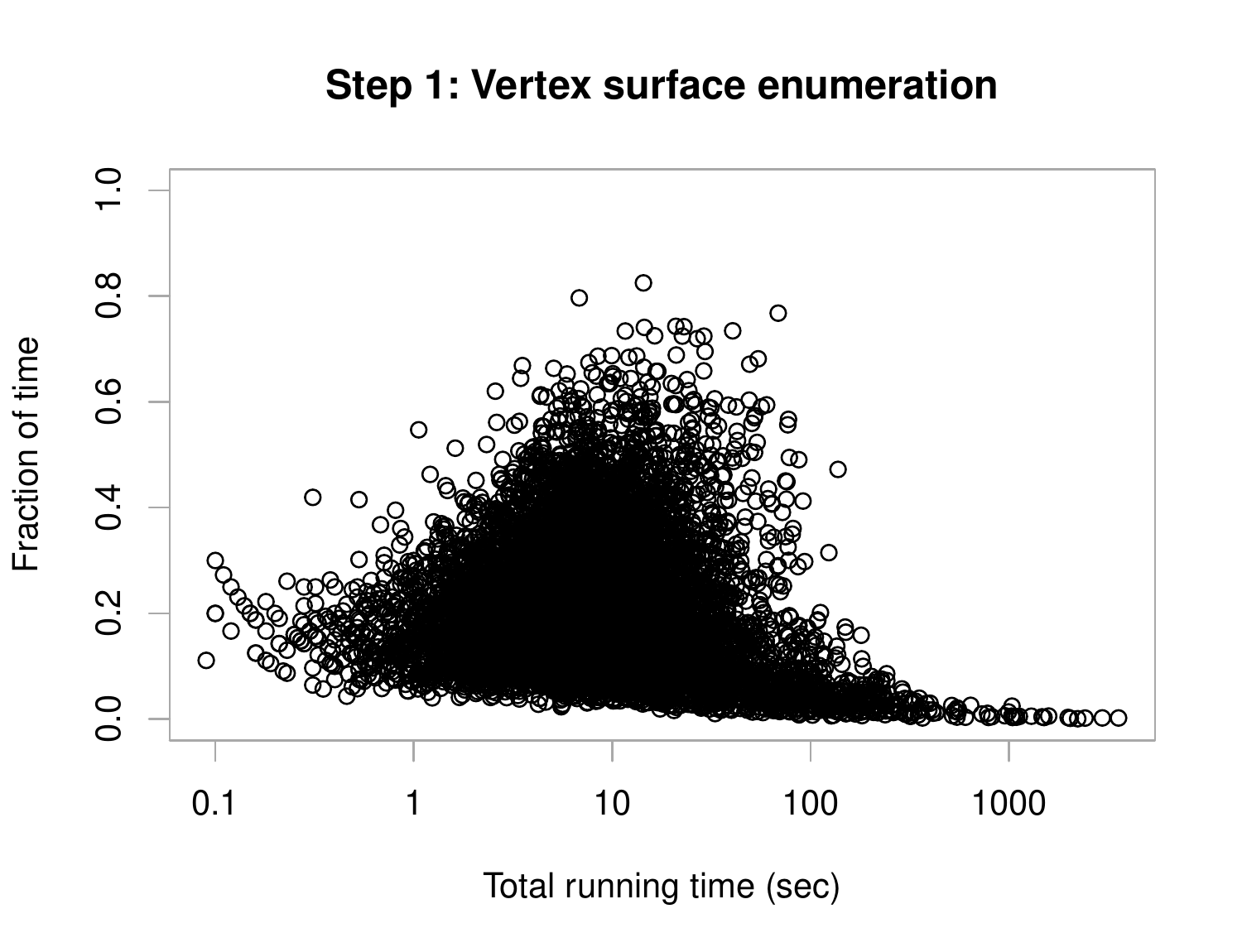}%
    \includegraphics[scale=0.5]{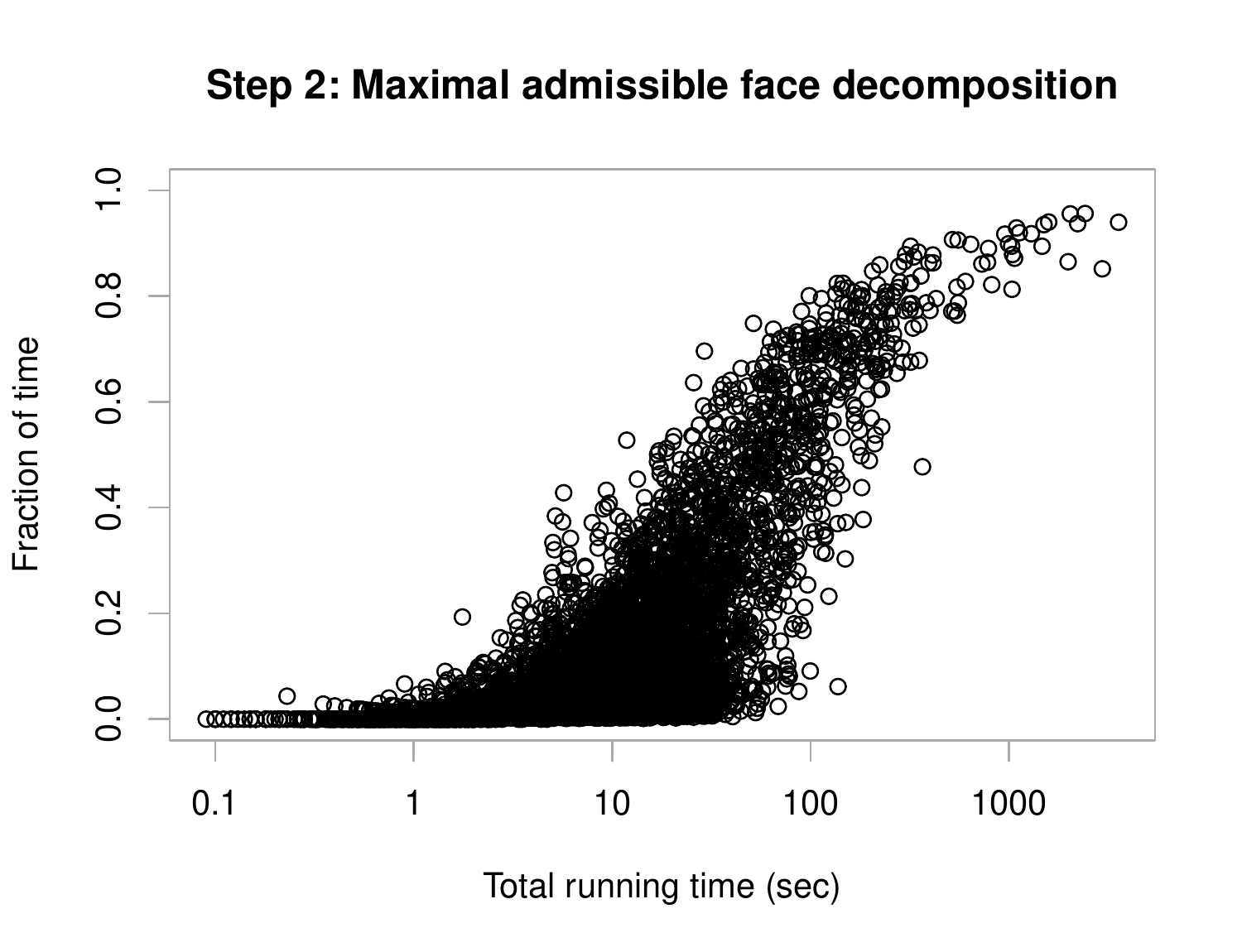} \\
    \includegraphics[scale=0.5]{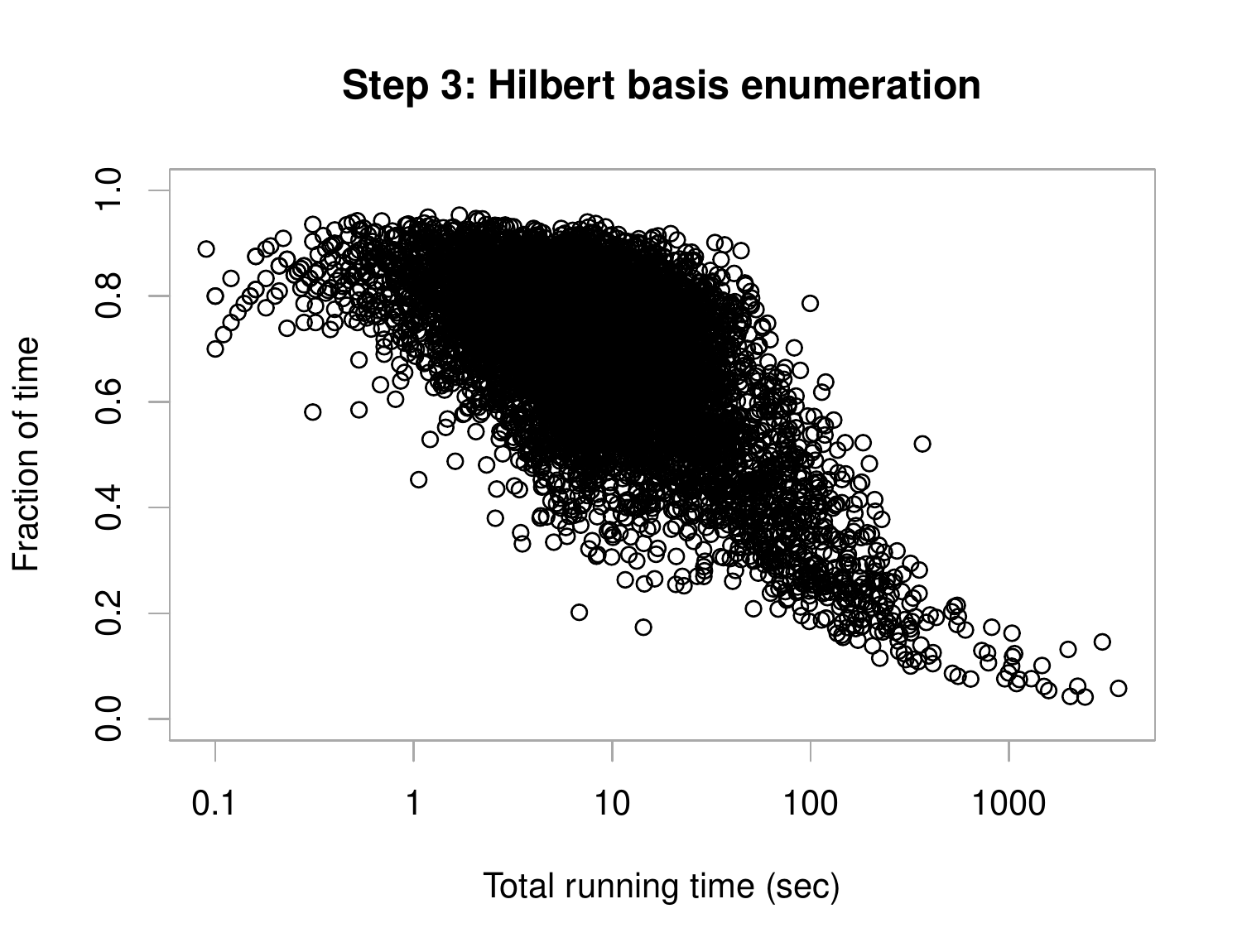}
    \caption{Division of time for the primal algorithm}
    \label{fig-perf-primal}
\end{figure}

We see that, as the running time grows, it is the
enumeration of maximal admissible faces that consumes the bulk of the
running time.  Significant gains could therefore be made by
incorporating more sophisticated algorithms and data structures into
this step.  One promising
approach could be to adapt the Kaibel-Pfetsch algorithm for
general polytopes \cite{kaibel02-lattice} to the setting of normal
surface theory.

%
%

\section{Application to crosscap numbers} \label{s-crosscap}

To illustrate the practicality of the primal algorithm, we use it to
compute previously-unknown crosscap numbers of knots.
Related to the genus of a knot, the
\emph{crosscap number} $\cc{K}$ of a knot $K$
is an invariant describing the smallest-genus
non-orientable surface that the knot bounds.

Crosscap numbers are difficult to compute: of the 2977
non-trivial prime knots with $\leq 12$ crossings,
2640 crosscap numbers are still unknown
(in contrast, the genus is known for all of these knots) \cite{www-knotinfo}.
Although there are specialised methods for certain classes of knots
\cite{adams12-spanning,hirasawa06-2bridge,ichihara10-pretzel,
teragaito04-crosscap},
there is no algorithm known for computing crosscap numbers in general.
However, we can come close: by enumerating fundamental normal
surfaces, it is possible to either compute the crosscap number of a knot
precisely or else reduce it to one of two possible values
\cite{burton12-crosscap}.

Due to space limitations we only give a very brief summary of the
computational results here.
See \cite{burton12-crosscap}
for more information on crosscap numbers and the underlying algorithm.

The crosscap number algorithm \cite[Algorithm~4.4]{burton12-crosscap} works
with triangulated \emph{knot complements}: these are triangulated 3-manifolds
with boundary, obtained by removing a small regular neighbourhood of a
knot from the 3-sphere.  The algorithm requires the triangulation $\tri$
to have specific properties (an ``efficient suitable triangulation''),
which can be tested by enumerating vertex normal surfaces.
It then enumerates all \emph{fundamental} normal surfaces in
$\tri$, runs some simple tests over each (such as whether the knot
bounds the surface, and computing the orientable or non-orientable genus),
and then either determines $\cc{K}$ precisely or declares that
$\cc{K} \in \{t, t+1\}$ for some $t$.

We process our knot complements in order by increasing number of
tetrahedra, and stop at the point where the computations become
infeasibly slow; our longest computation ran for $1.61$ days.
In total, we were able to run the algorithm over
585 knots, whose number of tetrahedra $n$ ranged from 5 to 27
(mean and median $n \simeq 22.7$ and $24$).
The mean and median running times were $\simeq 258$ and $\simeq 52$\ minutes
respectively, which is extremely pleasing given that these enumeration problems
involve up to $7 \times 27 = 189$ dimensions.  When this crosscap
number algorithm was developed \cite{burton12-crosscap},
such computations were believed to be beyond practical possibility
for all but the very simplest of knots.

The final outcomes are also pleasing.  We are able to compute
398 of the unknown crosscap numbers precisely, which
goes a significant way towards filling in the 2640 missing crosscap
numbers from the {\knotinfo} database.
For the remaining 187 calculations that completed,
we obtain information that is already known.
All results, both new and known,
are consistent with exiting {\knotinfo} data.\footnote{In fact,
    the crosscap number that we compute for the knot
    $8_{15}$ differs from {\knotinfo}, but matches the source from which the
    {\knotinfo} data was drawn \cite{adams12-spanning}.  This is presumably a
    transcription error in the database.}
A detailed table of results can be found at
\url{http://www.maths.uq.edu.au/~bab/code/}.

It should be noted that
the paper \cite{burton12-crosscap} also computes new
crosscap numbers, but using integer programming methods instead.
Our techniques here require knots whose complements do not
have too many tetrahedra;
in contrast, the integer programming techniques of
\cite{burton12-crosscap} can work with very large knot complements, but
only where good lower bounds on $\cc{K}$ are already known.
In this sense, the two methods complement each other well.

%
%

\section*{Acknowledgements}

The author is grateful to the Australian Research Council for their
support under the Discovery Projects funding scheme
(projects DP1094516 and DP110101104).
Computational resources used in this work
were provided by the Queensland Cyber Infrastructure Foundation.

%
%

\small
\bibliographystyle{amsplain}
\bibliography{pure}

%
%

\bigskip
\noindent
Benjamin A.~Burton \\
School of Mathematics and Physics, The University of Queensland \\
Brisbane QLD 4072, Australia \\
(bab@maths.uq.edu.au)

\end{document}